\documentclass[final]{siamltex}
\usepackage{epsfig}
\usepackage{amssymb}
\usepackage{amsmath, amssymb}
\usepackage{graphics,color}

\newtheorem{remark}[theorem]{Remark}

\def\IC{{\mathbb C}}
\def\IR{{\mathbb R}}

\def\bi{{\mathbf i}}
\def\bj{{\mathbf j}}

\def\({\left (}
\def\){\right )}

\def\Diag{{\rm diag}\,}
\def\tr{{\rm tr}\,}
\def\rank{{\rm rank}\,}

\def\M{M_{n_1}\otimes\cdots \otimes M_{n_m}}
\def\N{M_{n_1\cdots n_m}}
\def\A{A_1\otimes\cdots\otimes A_m}

\begin{document}
\openup 1\jot


\title{Linear maps preserving Ky Fan norms and Schatten  norms of tensor products of matrices}

\author{Ajda Fo\v sner\thanks{
Faculty of Management, University of Primorska,
Cankarjeva 5, SI-6104 Koper, Slovenia.
(Email: ajda.fosner@fm-kp.si)}
\and
Zejun Huang\thanks{
Department of Applied Mathematics,
The Hong Kong Polytechnic University,
Hung Hom, Hong Kong.
(Email: huangzejun@yahoo.cn)}
\and
Chi-Kwong Li\thanks{
Department of Mathematics, College of William and Mary, Williamsburg, VA 23187, USA;
Department of Mathematics, University of Hong Kong, Pokfulam, Hong Kong.
(Email: ckli@math.wm.edu)}
\and
Nung-Sing Sze\thanks{
Department of Applied Mathematics,
The Hong Kong Polytechnic University,
Hung Hom, Hong Kong.
(Email: raymond.sze@polyu.edu.hk)}
}

\maketitle

\begin{abstract}
For a positive integer $n$, let $M_n$ be the set of $n\times n$ complex matrices.
Suppose $\|\cdot\|$ is the Ky Fan $k$-norm with $1 \le k \le mn$ or
the Schatten $p$-norm with $1 \le p \le \infty$ ($p\ne 2$) on $M_{mn}$, where $m,n\ge 2$ are positive integers. 
It is shown that a linear map $\phi: M_{mn} \rightarrow M_{mn}$ satisfying
$$\|A\otimes B\| = \|\phi(A\otimes B)\| \quad \hbox{ for all } A \in M_m \hbox{ and } B \in M_n$$ 
if and only if there are unitary $U, V \in M_{mn}$ such that $\phi$
has the form $A\otimes B \mapsto U(\varphi_1(A) \otimes \varphi_2(B))V$, where
$\varphi_s(X)$ is either the identity map $X \mapsto X$ or the transposition map $X \mapsto X^t$.
The results are extended to tensor space $M_{n_1} \otimes \cdots \otimes M_{n_m}$
of higher level.  The connection of the problem to quantum information science is mentioned.
\end{abstract}

\begin{AMS}
{15A69, 15A86, 15A60, 15A18.}
\end{AMS}


\section{Introduction and preliminaries}

For a positive integer $n$, let $M_n$ be the set of $n\times n$ complex matrices. Now, suppose that $m,n\ge 2$ are positive integers. Then for $A \in M_m$ and $B \in M_n$, we denote by $A \otimes B \in M_{mn}$ their 
{\it tensor product} (a.k.a. the {\it Kronecker product}).
In many applied and pure studies, one  considers the tensor product
of matrices; for example, see \cite{B,Jain,M,W}. 
Most noticeably, the tensor product is often used
in quantum information science \cite{NC}.  
In a quantum system, quantum states are represented as density matrices (positive semi-definite
matrices with trace one). Suppose $A \in M_m$ and $B \in M_n$ are two quantum states in two quantum systems.
Then their tensor product $A\otimes B$ describes the joint state in the bipartite
system, in which the general states are density matrices in $M_{mn}$.
More generally, one may consider tensor states and general states in a multipartite system
$M_{n_1}\otimes \cdots \otimes M_{n_m}$ identified with $M_N$ where $N = \prod_{i=1}^m n_i$.

In general, it is relatively easy to construct and extract information
from matrices in tensor product form. For instance, the eigenvalues
(respectively, the singular values) of $A\otimes B$
have the form $a_ib_j$ with $1 \le i \le m$ and $1 \le j \le n$
if $A \in M_m$ and $B\in M_n$ have eigenvalues (respectively, singular values)
$a_1, \dots, a_m$ and $b_1, \dots, b_n$, respectively.
Thus, it is interesting to get information on the tensor space $M_{mn}$
by examining the properties of the small collection of matrices in tensor form $A\otimes B$.
In particular, if we consider a linear map $\phi: M_{mn} \rightarrow M_{mn}$ and
if one knows the images $\phi(A\otimes B)$ for $A \in M_m$ and $B\in M_n$, then
the map $\phi$ can be completely characterized 
as every $C \in M_{mn}$ is a linear combination of matrices in tensor form $A\otimes B$.
Nevertheless, the challenge is to use the limited information of the linear map
$\phi$ on matrices in tensor form to determine the structure of $\phi$.
In  \cite{FHLS}, we considered linear maps preserving the spectrum $\sigma(A\otimes B)$
and spectral radius $r(A\otimes B)$ of Hermitian matrices $A \in M_m$ and $B \in M_n$.
In \cite{J}, the author considered linear maps $\phi:M_{mn} \to M_{mn}$ satisfying 
$$\||A\otimes B\|| = \||\phi(A\otimes B)\|| \quad \hbox{ for all } A \in M_m \hbox{ and } B \in M_n,$$
where $\||\cdot\||$ is a certain (separability) norm defined in \cite{JK}.
This family of (separability) norms was shown to be related to the problem of detecting 
bounded entangled non-positive partial transpose states
and characterizing $k$-positive linear maps in quantum information science.

Suppose $X \in M_N$ has singular values $s_1(X) \ge \cdots \ge s_{N}(X)$.
The {\it Ky Fan $k$-norm} of $X$ is defined by
$$\|X\|_{(k)}=s_1(X)+\cdots+s_k(X).$$
The Ky Fan $1$-norm reduces to the {\it spectral norm} and the Ky Fan $N$-norm is also called
the {\it trace norm}.
For $p\geq 1 $, the {\it Schatten $p$-norm} of $X$ is defined by
$$\|X\|_p\equiv\left(\sum_{i=1}^{N} s_i(X)^p\right)^{1/p}.$$
The limiting case $p = \infty$   is just the spectral norm, $\|\cdot\|_1$ is
the   trace norm,  and $\|\cdot\|_2$ is the {\it Frobenius norm},
i.e., $\|X\|_2 = (\tr (XX^*))^{1/2}$.
In bipartite quantum systems, a well known criterion for separability of a state
is the computable cross norm (CCNR) criterion \cite{CW,R}, which
asserts that if a state (density matrix) $X$ in $M_{mn}$ is separable,
the trace norm of the realignment of $X$ (see e.g. \cite{CW}) is at most $1$.

The purpose of this paper is to study linear maps $\phi: M_{mn}\rightarrow M_{mn}$
satisfying
$$\|A\otimes B\| = \|\phi(A\otimes B)\| \quad \hbox{ for all } A \in M_m \hbox{ and } B \in M_n,$$
where $\|\cdot\|$ denotes the Ky Fan $k$-norm with $1 \le k \le mn$, or
the Schatten $p$-norm with $1 \le p \le \infty$. 

Note that even if we know that $\phi: M_{mn}\rightarrow M_{mn}$
is linear and satisfies
$\|A\otimes B\| = \|\phi(A\otimes B)\|$ for all $A \in M_m$ and $B \in M_n$,
it does not ensure that
$\|A_1\otimes B_1 + A_2\otimes B_2\| = \|\phi(A_1\otimes B_1) + \phi(A_2\otimes B_2)\|$
because $A_1\otimes B_1 + A_2\otimes B_2$ may not be of the form
$A \otimes B$. Thus, the proofs of our main results (Theorems \ref{T2.3} and \ref{T2.5}) are quite delicate as shown in the following discussion (in Section 2).
We will also extend the results to multipartite systems $M_{n_1}\otimes \cdots \otimes M_{n_m}$ in
Section 3. 

One may see \cite{A,CLS,GM,LP} and their references for some
background on linear preserver problems, and the preservers of the Ky Fan $k$-norms
and Schatten $p$-norms (without the tensor structure). It was shown in these papers that such norm preservers (except for the
 Schatten 2-norm preservers) $\phi: M_{n}\rightarrow M_{n}$ have the form
\begin{equation*}\label{eq1.1}
\phi(A)=UAV\quad {\rm or}\quad \phi(A)=UA^t V
\end{equation*}
for some unitary matrices  $U,V\in M_n$, 
where $A^t$ is the transpose of $A$.
 One can also see
\cite{FHLS, FLPS, J, LPS11} and their references for some recent results on linear preserver
problems on tensor spaces arising in quantum information science.

In our discussion, we will use $X^t$ and $X^*$ to denote the transpose and the conjugate transpose
of a square matrix $X$, respectively. For any $A \in M_m$ and $B \in M_n$,
we denote by $A\oplus B$ their direct sum.
The $n\times n$
identity matrix will be denoted by $I_n$. Denote by $E_{ij}$ the square matrix which the
$(i,j)$-entry is equal to one and all the others are equal to zero, where the size of
$E_{ij}$ should be clear in the context.

\section{Bipartite systems}

\subsection{Spectral norm}
In what follows we denote by $\|\cdot\|$ the spectral norm.
Recall that the spectral norm is the same as Ky Fan $1$-norm.
We first present the result for spectral norm.

\medskip
\begin{theorem}
\label{T1}
The following are equivalent for a linear map $\phi: M_{mn} \rightarrow M_{mn}$.
\begin{enumerate}
\item[\rm (a)] $\|\phi(A\otimes B)\| = \|A\otimes B\|$ for all $A \in M_m$ and $B \in M_n$.

\item[\rm (b)] There are unitary matrices $U, V \in M_{mn}$ such that
$$\phi(A\otimes B) = U(\varphi_1(A) \otimes \varphi_2(B))V \quad \hbox{for all $A \in M_m$ and $B \in M_n$,}$$
where $\varphi_s$ is the identity map or the transposition map $X \mapsto X^t$ for $s=1,2$.
\end{enumerate}
\end{theorem}

\begin{proof}
The implication (b) $\Rightarrow$ (a) is obvious. 
Conversely, assume that $\|\phi(A\otimes B)\| = \|A\otimes B\|$ for all $A \in M_m$ and $B \in M_n$. 
In the following, we first show that $\phi$ maps the matrix $E_{ii}\otimes E_{jj}$ 
to $U (E_{ii} \otimes E_{jj})V$ for all $1\le i \le m$ and $1\le j \le n$,
where $U$ and $V$ are some unitary matrices.

Let $1\le i\le m$ and $1\le j,s\le n$ with $j\ne s$. Then $\phi(E_{ii} \otimes E_{jj})$ has norm one and the same is true for $\phi(E_{ii} \otimes E_{jj} + \gamma E_{ii}\otimes E_{ss})$ whenever $|\gamma| \le 1$.  
Suppose $x_j, y_j \in \IC^{mn}$ are the left and right norm attaining unit vectors of $\phi(E_{ii} \otimes E_{jj})$, i.e., $x_j^* \phi(E_{ii} \otimes E_{jj}) y_j = 1$. Then for any unitary matrices $X_j$ and $Y_j$ with $x_j$ 
and $y_j$ as their first columns, we have
$X_j^* \phi(E_{ii} \otimes E_{jj}) Y_j = [1] \oplus G_j$ for some $G_j \in M_{mn-1}$
with $\|G_j\| \le 1$. Since $\phi(E_{ii} \otimes E_{jj} + \gamma E_{ii} \otimes E_{ss})$ has norm one for all $|\gamma|\le 1$, the matrix $X_j^* \phi(E_{ii} \otimes E_{ss}) Y_j$ must have the form $[0] \oplus G_s$ for some $G_s \in M_{mn-1}$. It follows that the left and right norm attaining vectors of $\phi(E_{ii} \otimes E_{ss})$, say $x_s$ and $y_s$, must be orthogonal to $x_j$ and $y_j$ respectively. As $j$ and $s$ are arbitrary, it follows that
$\{x_1,\dots,x_n\}$ and $\{y_1,\dots,y_n\}$ are orthonormal sets.

Let $U_i$ be an $mn\times mn$ unitary matrix with $x_1,\dots,x_n$ as its first $n$ columns
and $V_i$ be an $mn\times mn$ unitary matrix with $y_1^*,\dots,y_n^*$ as its first $n$ rows.
From the above discussion, one has
$$\phi(E_{ii} \otimes E_{jj}) = U_i (E_{jj} \oplus P_{ij}) V_i\quad\hbox{for all } j = 1,\dots,n$$
for some $P_{ij} \in M_{mn-n}$ and hence
$$\phi(E_{ii} \otimes D) = U_i (D \oplus P_{i,D}) V_i\quad \hbox{for any diagonal matrix }D \in M_n$$
for some $P_{i,D} \in M_{mn-n}$. 
Note also that if $1\le i,r\le m$ with $i \ne r$,
then $\phi(E_{ii} \otimes D + \gamma E_{rr} \otimes D)$ has norm one for any diagonal unitary matrix $D \in M_n$ and any scalar $\gamma$ with $|\gamma| \le 1$.
Consequently, we see that the left and right norm attaining vectors of $\phi(E_{ii}\otimes D)$ and those of $\phi(E_{rr}\otimes D)$ are orthogonal.
By a similar argument, there are unitary $U, V \in M_{mn}$ such that
$$\phi(E_{ii}\otimes D) = U(E_{ii}\otimes D) V\quad\hbox{for any $1\le i\le m$ and unitary diagonal $D\in M_n$.}$$
In particular,
$$\phi(E_{ii}\otimes E_{jj}) = U(E_{ii}\otimes E_{jj})V\quad \hbox{for $1\le i\le m$ and $1\le j\le n$.}$$

For the sake of the simplicity, we assume that $U$ and $V$ are identity matrices. 
Next, we show that $\phi(E_{ii} \otimes B) = E_{ii} \otimes \varphi_i(B)$
for all $B\in M_n$, where $\varphi_i$ is a linear map on $M_n$.
For any unitary $Y \in M_n$ and $1\le i\le m$, $1\le j\le n$ we can apply the argument
in the preceding paragraphs to conclude that $\phi(E_{ii} \otimes YE_{jj}Y^*)$ has rank one. Since 
$$\|\phi(E_{ii} \otimes I_n) + \gamma \phi(E_{ii}\otimes YE_{jj}Y^*)\|  = \|\phi(E_{ii} \otimes I_n + \gamma E_{ii}\otimes YE_{jj}Y^*)\| = 1+\gamma$$
for any positive scalar $\gamma$, the left and right norm attaining vectors of $\phi(E_{ii}\otimes YE_{jj}Y^*)$
are also left and right norm attaining vectors of $\phi(E_{ii} \otimes I_n)=E_{ii} \otimes I_n$. Thus, $\phi(E_{ii} \otimes YE_{jj}Y^*)=E_{ii} \otimes Z$ for some rank one $Z\in M_n$. Since this is true
for any unitary $Y \in M_n$, by linearity of $\phi$, 
we conclude that there exists a linear map $\varphi_i:M_n \rightarrow M_n$ such that
$$\phi(E_{ii}\otimes B) = E_{ii}\otimes \varphi_i(B) 
\quad \hbox{for all $B\in M_n$}.$$  
Clearly, $\varphi_i$ preserves the spectral norm. 
Therefore, $\varphi_i$ has the form $X \mapsto W_iX\tilde{W}_i $ or $X \mapsto W_iX^t\tilde{W}_i $ for some unitary
$W_i,\tilde{W}_i \in M_n$ (e.g., see \cite{CLS} and its references). For simplicity, we may assume that 
$W_i =\tilde{W}_i= I_n$ for all $i = 1, \dots, m$.

Let $X\in M_m$ be any unitary matrix. Repeating the same argument as above, one can show that $$\phi(XE_{ii}X^* \otimes B)= U_X(E_{ii} \otimes \varphi_{i,X}(B))V_X$$
for $1\le i \le m$ and $B \in M_n$, where $U_X, V_X\in M_{mn}$ are unitary matrices depending on $X$ and $\varphi_{i,X}:M_n\to M_n$ is either the identity map or
the transposition map depending on $i$ and $X$. Moreover, since $\phi(I_{mn}) = I_{mn}$, we have $V_{X} = U_{X} ^*$.

Now we show that all the maps $\varphi_{i,X}$ are the same.
For any real symmetric $S\in M_n$ and any unitary $X \in M_m$ we have
\begin{eqnarray*}
  \phi\left( I_m \otimes S \right)
 =  \phi\left(\sum_{i=1}^m XE_{ii}X^* \otimes S \right)
  = U_X \left(\sum_{i=1}^m XE_{ii}X^* \otimes S\right)U_X^*
 =  U_X \left( I_m \otimes S \right)U_X^*.
\end{eqnarray*}
In particular, when $X=I_m$, we have $\phi\left( I_m \otimes S \right)=I_m \otimes S$.
Thus, $U_X \left( I_m \otimes S \right)U_X^*=I_m \otimes S$ and this yields that $U_X$ commutes with $I_m \otimes S$ for all real symmetric $S$. Hence, $U_X$ has the form $W_X \otimes I_n$ for some unitary $W_X \in M_m$ and
$$\phi\left(XE_{ii}X^* \otimes B \right) = (W_X E_{ii} W_X^*) \otimes \varphi_{i,X}(B) \quad\hbox{ for }\quad \hbox{$1\le i \le m$ and $B \in M_n$}.$$
Now, consider the linear maps 
${\rm tr}_1: M_{mn} \to M_n$ and  ${\rm Tr}_1: M_{mn} \to M_n$ defined by
$${\rm tr}_1(A\otimes B) = (\tr A)B \quad \hbox{ and } \quad {\rm Tr}_1(A\otimes B) = \tr_1\left( \phi(A\otimes B) \right)$$ for all
$A\in M_m$ and $B \in M_n$.
Notice that the map ${\rm tr}_1$ is known as the partial trace function
in quantum information science context.
Then $${\rm Tr}_1 \left( \phi \left(XE_{ii}X^* \otimes B \right) \right) = \varphi_{i,X}(B).$$ 
So, ${\rm Tr}_1$ induces a map $XE_{ii}X^* \mapsto \varphi_{i,X}$,
where $\varphi_{i,X}$ is either the identity map or the transpose map.
Note that ${\rm Tr}_1$ is linear and therefore continuous,  and the set 
$$\{XE_{ii}X^*: 1 \le i \le m, X^*X = I_m \} = \{ xx^* \in M_m: x^*x = 1\}$$
is connected.  So, all the maps $\varphi_{i,X}$ have to be the same. Replacing $\phi$ by the map
$A\otimes B \mapsto \phi(A\otimes B^t)$, if necessary, we may assume that this common map is the identity map. Next, using the linearity of $\phi$,
one can conclude that for every $A \in M_m$ and $B \in M_n$ we have  $$\phi \left(A \otimes B \right) = \varphi_1(A) \otimes B,$$
where $\varphi_1(A) \in M_m$ depends on $A$ only. Recall that $\varphi_1:M_m\to M_m$ is a linear map and $\|\varphi_1(A)\| = \|A\|$
for all $A \in M_m$. Hence, $\varphi_1$ has the form $A\mapsto UAV $ or $A\mapsto UA^tV $ for some unitary $U, V\in M_m$. This completes the proof.
\end{proof}

\subsection{Ky Fan $k$-norms}
We now turn to Ky Fan $k$-norms. 
Two matrices $A,B\in M_n$ are called {\it orthogonal} if $AB^*=A^*B=0$ (see \cite{LSS}).  We write $A\, \bot\, B$ to indicate that $A$ and $B$ are orthogonal. It is shown in \cite{LSS} that $A\, \bot\, B$ if and only if there are unitary matrices $U,V\in M_n$ such that $UAV=\Diag(a_1,\ldots,a_n)$ and $UBV=\Diag(b_1,\ldots,b_n)$ with $a_i,b_i\geq 0$ and $a_ib_i=0$ for $i=1,\dots,n$. The matrices $A_1,\ldots,A_t$ are said to be {\it pairwise orthogonal} if $A_i^*A_j=A_iA_j^*=0$ for any distinct $i,j\in \{1,\ldots,t\}$. In this case, there are unitary matrices $U,V\in M_n$ such  that $UA_iV=D_i$ for $i=1,\dots,t$ with each $D_i$ being nonnegative diagonal matrix and $D_iD_j=0$ for any distinct $i,j\in \{1,\ldots,t\}$.

We have the following lemmas relating to orthogonality,
which are useful in the proof of Ky Fan $k$-norm results (Theorems \ref{T2.3} and \ref{T4}).

\medskip
\begin{lemma}\label{le1}{\rm \cite{LSS}}
Let $A,B\in M_n$ be nonzero matrices. Then 
$$\|\alpha A+\beta B\|_{(k)}=|\alpha|\| A \|_{(k)}+|\beta|\| B \|_{(k)}$$ 
for every pair of complex numbers $\alpha$ and $\beta$  if and only if $A\,\bot\, B$ and $\rank A+\rank B\leq k$.
\end{lemma}

Denote by $\sigma(A)$ the spectrum of a matrix $A\in M_n$. Using the same arguments as in the proof of Lemma 2 in \cite{LSS}, we have the following result.
(One can also see \cite[p.468, Problem 3]{HJ}  for part (a) of Lemma \ref{le2.3}.)

\medskip
\begin{lemma}\label{le2.3}
Let $A\in M_n$ be positive semidefinite and let $B\in M_n$ be Hermitian.
\begin{enumerate}
\item[\rm (a)]  $\sigma(AB)\subseteq  \mathbb{R}$.
\item[\rm (b)] If $\sigma(AB)=\{0\}$, then there exists a unitary $U\in M_n$ such that
$$UAU^*= \begin{bmatrix} A_1 & 0 \cr 0 & 0 \end{bmatrix} \quad\hbox{and}\quad UBU^*=\begin{bmatrix}0&X\\X^*&B_1\end{bmatrix},$$
 where $A_1\in M_s$ is invertible and $B_1\in M_{n-s}$ with $0\leq s\leq n$.
\end{enumerate}
\end{lemma}

\medskip
\begin{lemma}\label{le2}
Let $1\le k\le n$ and $A, B \in M_n$ with spectral norm at most $1$.
Suppose
$$\rank A \le k, \quad \|A+\alpha B\|_{(k)}=k\quad\hbox{and}\quad
\| 2A+\alpha B\|_{(k)}=  \| A\|_{(k)}+\| A+\alpha B\|_{(k)}$$
for any unit complex number $ \alpha $.
Then $A\,\bot\, B$ and  $\sigma(A^*A)\subseteq\{0,1\}.$
\end{lemma}

\begin{proof}
First, let $\alpha=1$ and suppose $2A+B$  has singular value decomposition $2A+B=U_1DV_1$ 
with $D=\Diag(d_1,\ldots,d_n)$
and $d_1\geq\cdots\geq d_n\ge 0$. Then
$$D=U_1^*(2A+B)V_1^*=U_1^*AV_1^*+U_1^* (A+B)V_1^*.$$
Denote the diagonal entries of $U_1^*AV_1^*$ and $U_1^*BV_1^*$ by $a_1,\ldots,a_n$ and $b_1,\ldots,b_n$, respectively. Then we have
\begin{eqnarray*}
 \sum_{i=1}^k|a_i|\leq \|A\|_{(k)},\quad
 \sum_{i=1}^k|a_i+b_i|\leq \|A+B\|_{(k)}, \quad
\sum_{i=1}^k(2a_i+b_i)= \sum_{i=1}^kd_i=\|2A+B\|_{(k)}.
\end{eqnarray*}
It follows that
$$\sum_{i=1}^k \left( a_i + (a_i + b_i) \right) \le \sum_{i=1}^k \left( |a_i| + |a_i+b_i| \right) \le \|A\|_{(k)} + \|A+B\|_{(k)}
= \|2A+B\|_{(k)} = \sum_{i=1}^k (2a_i+b_i),$$
and so all the above inequalities are indeed equalities,  
which ensure that $a_i = |a_i|$ and $a_i+b_i = |a_i+b_i|$ are nonnegative real numbers with
$$\sum_{i=1}^ka_i = \sum_{i=1}^k |a_i| = \|A\|_{(k)}
\quad\hbox{and}\quad
\sum_{i=1}^k(a_i+b_i)=\sum_{i=1}^k |a_i+b_i| = \|A+B\|_{(k)}.$$  By Theorem 3.1 of \cite{Li}, we have  $$U_1^*AV_1^*=A_{1}\oplus A_{2}\quad {\rm and}\quad  U_1^*(A+B)V_1^*=(A_1+B_{1})\oplus (A_2+B_{2}),$$ where $A_{1}$ and $A_1+B_1$ are $k\times k$ positive semidefinite  matrices and
\begin{equation*}\label{eq.2}
\|A \|_{(k)}=\|A_1\|_{(k)}=\tr A_1,\quad\|A+B \|_{(k)}= \|A_1+B_1\|_{(k)}=\tr(A_1+B_1).
\end{equation*}
Since $\rank A\leq k$, it follows that $A_2=0$ and we may assume that $B_2=\Diag(d_{k+1},\ldots,d_n)$. 
Without loss of generality, we also assume that $U_1=V_1=I_n$ and
$$ A=A_{1}\oplus 0, \qquad B= B_{1} \oplus \Diag(d_{k+1},\ldots,d_n).$$
Now take a unit $\alpha_1\ne \pm 1$. Suppose  the singular value decomposition of $2A_1+\alpha_1 B_1$ is $2A_1+\alpha_1 B_1=U_2D_1V_2$ with $D_1=\Diag(\beta_1,\ldots,\beta_k)$ and $\beta_1\geq \cdots\geq \beta_k\geq 0$. Let $U_3=U_2^*\oplus I_{n-k}$ and $V_3=V_2^*\oplus \alpha_1^{-1}I_{n-k}$. 
Set $A_3 = U_2^*A_1V_2^*$ and $B_3 = \alpha_1 U_2^*B_1V_2^*$.
Then
\begin{eqnarray*}
\Diag(\beta_1,\ldots,\beta_k,d_{k+1},\ldots,d_{n})&=&U_3(2A+\alpha_1B)V_3 \\
&=&U_2^*A_1V_2^*\oplus 0_{n-k}+U_2^*(A_1+\alpha_1 B_1)V_2^*\oplus \Diag(d_{k+1},\ldots,d_{n}) \\
&=& A_3\oplus 0_{n-k}+(A_3+B_3)\oplus \Diag(d_{k+1},\ldots,d_{n}).
\end{eqnarray*}
First, assume that $\|2A+\alpha_1 B\|_{(k)}=\sum_{i=1}^k\beta_i$. Using the same argument as above, we see that
$A_3$ and $A_3+B_3$ are $k\times k$ positive semidefinite matrices with $\|A \|_{(k)}=\tr A_3$ and $\|A+\alpha_1B \|_{(k)}=\tr (A_3+B_3)$.
By Lemma \ref{le2.3}(a), $\sigma(A_1B_1) \subseteq \IR$ and $\sigma(A_3 B_3) \subseteq \IR$. Also
$$\sigma(\alpha_1 A_3 B_3) 
= \sigma(\alpha_1 A_3 B_3^*) 
= \sigma(\alpha_1 (U_2^* A_1 V_2^*)(\alpha_1 U_2^* B_1 V_2)^* )
= \sigma(U_2^*A_1 B_1 U_2)
= \sigma(A_1 B_1) \subseteq \IR.$$
This implies $\sigma(A_1B_1) = \sigma(A_3 B_3) = \{0\}$.
According to Lemma \ref{le2.3}(b), we know that there exists a unitary $U_4$ such that $$U_4A_3U_4^*=A_4\oplus 0,\qquad U_4B_3U_4^*=\begin{bmatrix}0&X\\X^*&B_4\end{bmatrix},$$
where $A_4\in M_s$ is positive definite for some $0\leq s\leq k$, $B_4\in M_{k-s}$, and $\|A_3+B_3\|_{(k)}=\tr A_4+
\tr B_4$. Since both $A$ and $B$ have spectral norm at most $1$, the diagonal entries of $A_4$ and $B_4$ are less than or equal to one. So  $\tr (A_3+B_3)=\|A_3+B_3\|_{(k)}=\|A+\alpha_1B\|_{(k)}=k$ ensures that all the diagonal entries of $A_4$ and $B_4$ are equal to one and that the singular values of $A_3+B_3$ are all equal to one. It follows that $X=0$, $A_4=I_s$, and $B_4= I_{k-s}$, which implies $A\bot B$ and $\sigma(A^*A)\subseteq\{0,1\}.$

Now, assume that $\|2A+\alpha_1 B\|_{(k)}\ne\sum_{i=1}^k\beta_i$ and
suppose that the largest $k$ singular values of $2A+\alpha_1B$ are $\beta_1,\ldots,\beta_r,d_{k+1},\ldots,d_{k+s}$ with $r+s=k$ and $r<k$. Denote by $\tilde{A}=A_3\oplus 0$ and $\tilde{B}=B_3\oplus\Diag(d_{k+1},\ldots, d_{n})$. Applying Theorem 3.1 of \cite{Li} again, it follows that
$$ \tilde{A} =\tilde{A}[1,\ldots,r]\oplus 0_{n-r}\geq 0 
\quad \hbox{and}\quad
  \tilde{B} = \tilde{B}[1,\ldots,r]\oplus \Diag(\beta_{r+1},\ldots,\beta_k,d_{k+1},\ldots,d_n)\geq 0, $$ where $\tilde{A}[ 1,\ldots,r]$ and $\tilde{B}[1,\ldots,r]$ are the principal submatrices of $\tilde{A}$ and $\tilde{B}$ indexed by $1,\ldots,r$.  Without loss of generality, we can assume
$$A=\tilde{A}[1,\ldots,r]\oplus 0_{n-r}
\quad \hbox{and}\quad
B = \tilde{B}[1,\ldots,r]\oplus \Diag(\beta_{r+1},\ldots,\beta_k,d_{k+1},\ldots,d_n).$$
Choosing a unit $\alpha_2\ne \{\pm 1,\pm\alpha_1\}$ and repeating the above process at most $n-2$ times, we can show that $A\bot B$ and $\sigma(A^*A)\subseteq\{0,1\}.$
\end{proof}

\begin{lemma}\label{2.5}
Let $1\le k \le n$ and $A,B \in M_n$ such that $A$ has rank at least $k$ and $\sigma(A^*A) \subseteq \{0,1\}$.
If
$$\|A + \alpha B\|_{(k)} = k$$
for all unit complex numbers $\alpha$, then $A \,\bot\, B$.
\end{lemma}

\begin{proof}
Without loss of generality, we may assume that $A = I_s \oplus 0$ with $s \ge k$.
Denote $A = \left[a_{ij}\right]$ and $B = \left[b_{ij}\right]$.
We claim that for any arbitrary and fixed $1\le p_1 < \cdots < p_k \le s$,
$b_{ij} = 0$ whenever $\{i,j\} \cap \{p_1,\dots,p_k\} \ne \emptyset$.
If the claim holds, then $B$ has the from $0_s \oplus B_2$ and therefore $A\, \bot\, B$.
To prove the claim, it suffices to show the case when $(p_1,\dots,p_k) = (1,\dots,k)$.
Notice that 
$$k = \|A + \alpha B\|_{(k)} \ge \sum_{j=1}^k \left|a_{jj} + \alpha b_{jj}\right|
\ge \left| \sum_{j=1}^k (a_{jj} + \alpha b_{jj})\right|
= \left| k + \alpha \sum_{j=1}^k b_{jj}\right|.$$
The two equalities holds for all complex unit $\alpha$ if and only if $b_{jj} = 0$ for all $1\le j \le k$.
Thus, we have $\|A+\alpha B\|_{(k)} = k = \sum_{j=1}^k |a_{jj} + \alpha b_{jj}|$
for all complex unit $\alpha$.
Now by \cite[Theorem 3.1]{Li},
$$A + \alpha B = C_\alpha \oplus D_\alpha$$
where $C_\alpha$ is positive semidefinite and $\|C_\alpha\|_{(k)} = k$.
It follows that $B = B_1 \oplus B_2$ with $B_1 \in M_k$.
Notice that $C_\alpha = I_k + \alpha B_1$ is positive semidefinite for all complex unit $\alpha$.
Then $B_1$ has to be the zero matrix, i.e., $B = 0_k \oplus B_2$.
Therefore, the claim holds and the result follows.
\end{proof}

\medskip
We now present of the result for the Ky Fan $k$-norm.
\begin{theorem}
\label{T2.3}
Let $m,n\geq 2$ and $2\leq k\leq mn$ and $\phi: M_{mn} \rightarrow M_{mn}$ be a linear map.
The following are equivalent.
\begin{enumerate}
\item[\rm (a)] $\|\phi(A\otimes B)\|_{(k)} = \|A\otimes B\|_{(k)}$ for all $A \in M_m$ and $B \in M_n$.

\item[\rm (b)] There are unitary matrices $U, V \in M_{mn}$ such that
$$\phi(A\otimes B) = U(\varphi_1(A) \otimes \varphi_2(B))V \quad\hbox{for all}\quad A \in M_m \hbox{ and }B \in M_n,$$
where $\varphi_s$ is the identity map or the transposition map $X \mapsto X^t$ for $s=1,2$.
\end{enumerate}\end{theorem}

\begin{proof}
The implication (b) $\Rightarrow$ (a) is obvious. Conversely, assume that $\|\phi(A\otimes B)\|_{(k)} = \|A\otimes B\|_{(k)}$ for  all $A \in M_m$ and $B \in M_n$.  We assert  that there exist unitary $U,V\in M_{mn}$ such that
\begin{equation}\label{eq.4}
\phi(E_{ii}\otimes E_{jj}) = U(E_{ii}\otimes E_{jj})V\quad {\rm for}~  1\le i\le m  ~{\rm and}~  1\le j\le n.
 \end{equation}
It suffices to show that
\begin{equation}\label{eq.5}
\phi(E_{ii}\otimes E_{jj})\, \bot\,  \phi(E_{rr}\otimes E_{ss})
\end{equation}
 for any distinct pairs $(i,j)$ and $(r,s)$ with $1\leq i,r\leq m$ and $1\leq j,s\leq n$.
We distinguish three cases.

\vspace{0,2cm}

\noindent
{\it Case 1.} Suppose $i=r$ or $j=s$. We have $$ \|\alpha\phi(E_{ii}\otimes E_{jj})+ \beta\phi(E_{rr}\otimes E_{ss})\|_{(k)}=|\alpha|\|\phi(E_{ii}\otimes E_{jj}) \|_{(k)}+|\beta|\| \phi(E_{rr}\otimes E_{ss})\|_{(k)}$$
for all complex numbers $\alpha$ and $\beta$. Applying Lemma \ref{le1}, we have (\ref{eq.5}) and
\begin{equation}\label{eq5}
\rank \phi(E_{ii}\otimes E_{jj})+\rank \phi(E_{rr}\otimes E_{ss})\leq k.
\end{equation}

\vspace{0,2cm}

\noindent
{\it Case 2.} Suppose $i\ne r$ and $j\ne s$.
Let
$$G=\phi( E_{ii} \otimes (E_{jj} +   E_{ss}))
\quad {\rm and}\quad   H=\phi( E_{rr} \otimes (E_{jj} +   E_{ss})).$$ 

\vspace{0,1cm}

{\it Subcase 2.a.} Assume first that $k\le3$. By Case $1$, $\phi( E_{ii} \otimes E_{jj})$ and
$\phi( E_{ii} \otimes E_{ss})$ are orthogonal with $\|\phi( E_{ii} \otimes E_{jj})\|_{(k)}=
\|\phi( E_{ii} \otimes E_{ss})\|_{(k)}=1$. So, $\|G\| \le 1$ and $\rank G\leq k$. Similarly, 
$\|H\| \le 1$ and $\rank H\leq k$. Furthermore, $\|G+\gamma H\|_{(k)}=k$,   and 
$\|2G+\gamma H\|_{(k)}=\| G \|_{(k)}+\| G+\gamma H\|_{(k)}$
for all complex units $\gamma$.  Applying Lemma \ref{le2}, we get $G\, \bot\, H$.
Because $\phi(E_{ii} \otimes E_{jj})\, \bot\, \phi(E_{ii} \otimes E_{ss})$
and $\phi(E_{rr} \otimes E_{jj})\, \bot\, \phi(E_{rr} \otimes E_{ss})$,
we have (\ref{eq.5}).

\vspace{0,1cm}

{\it Subcase 2.b.} Now, suppose that $k > 3$. We have
$$\| \alpha G+\beta H \|_{(k)}=|\alpha|\|G\|_{(k)}+ |\beta|\|H\|_{(k)} $$
for all complex numbers $\alpha,\beta$.
Applying Lemma \ref{le1} again, we get
 $G\,\bot\, H$ and, hence, (\ref{eq.5}) follows.

\vspace{0,2cm}

From above, we showed that (\ref{eq.4}) holds.
For the sake of the simplicity, we assume that $U$ and $V$ are identity matrices. Now, for any unitary $Y \in M_n$ and $1\le i\le m$, $1\le j\le n$ we can apply the argument
in the preceding paragraphs to conclude that $\phi(E_{ii} \otimes YE_{jj}Y^*)$ has rank one and it is orthogonal to $\phi(E_{rr} \otimes YE_{ss}Y^*)$ for any distinct pairs $(i,j)$ and $(r,s)$. 
It follows that
 $$\|\phi(E_{ii} \otimes I_n) + \gamma \phi(E_{ii}\otimes YE_{jj}Y^*)\| = \|\phi(E_{ii} \otimes I_n + \gamma E_{ii}\otimes YE_{jj}Y^*)\| = 1+\gamma$$
for any positive scalar $\gamma$. Thus, $\phi(E_{ii} \otimes YE_{jj}Y^*)=E_{ii} \otimes Z$ for some rank one $Z\in M_n$. Since this is true
for any $1\le i \le m$ and unitary $Y \in M_n$, we conclude that there exists a linear map $\varphi_i:M_n \rightarrow M_n$ such that
\begin{equation*}
\phi(E_{ii}\otimes B) = E_{ii}\otimes \varphi_i(B)
 \end{equation*}
for all complex matrices $B\in M_n$.  Clearly, $\varphi_i$ preserves the Ky Fan $k$-norm of all $B\in M_n$ and
 $\varphi_i(E_{jj}) = E_{jj}$ for all $1\le j\le n$. (Here the Ky Fan $k$-norm reduces to the trace norm if $k\ge n$.)
Hence, $\varphi_i$ has the form $X \mapsto W_iX\tilde{W}_i $ or $X \mapsto W_iX^t\tilde{W}_i $ for some unitary
$W_i,\tilde{W}_i \in M_n$. Now, we can adapt the arguments in the last two paragraphs in the proof of Theorem \ref{T1}
to obtain our conclusion.
\end{proof}

\subsection{Schatten $p$-norms}
We now study the linear preserver for Schatten $p$-norms. We first present a key lemma of the result.

\medskip
\begin{lemma}\label{le3}
{\rm \cite{CM}} 
Let $T,S\in M_n$. Then
\begin{enumerate}
\item[\rm (i)] $2^{p-1}(\|T\|_p^p+\|S\|_p^p)\leq \|T+S\|_p^p+\|T-S\|_p^p\leq 2(\|T\|_p^p+\|S\|_p^p)$ if $1\leq p\leq 2$,

\item[\rm (ii)] $2(\|T\|_p^p+\|S\|_p^p)\leq \|T+S\|_p^p+\|T-S\|_p^p\leq 2^{p-1}(\|T\|_p^p+\|S\|_p^p)$ if $2\leq p<\infty$.
\end{enumerate}
If $p=2$, all equalities always hold; if $p\ne 2$, any equality holds if and only if $T^*TS^*S = S^*ST^*T =0$.

\end{lemma}

\medskip
\begin{remark}\label{rem7}\em
Suppose $p\ne 2$ and the singular values of $T$ and $S$ are $\{t_1,\ldots,t_n\}$ and $\{s_1,\ldots,s_n\}$, respectively. 
If any of equality in Lemma \ref{le3} holds, then $T^*TS^*S=S^*ST^*T=0$ implies $T^*T\bot S^*S$. So there exists unitary $U\in M_n$ such that $U^*T^*TU=\Diag(t_1^2,\ldots,t_n^2)$ and $U^*S^*SU=\Diag(s_1^2,\ldots,s_n^2)$ with $s_it_i=0$ for $i=1,\ldots,n$.  Replacing $T$ and $S$ with $T^*$ and $S^*$, we get $TT^*SS^*=0$ and there exists unitary $V\in M_n$ such that $V^*TT^*V=\Diag(t_1^2,\ldots,t_n^2)$ and $V^*SS^*V=\Diag(s_1^2,\ldots,s_n^2)$. It follows that $V^*TU=\Diag( t_1,\ldots, t_n )$ and $V^*SU=\Diag( s_1,\ldots, s_n )$, which implies $T\, \bot\, S$.
\end{remark}

\medskip
\begin{theorem}
\label{T2.5}
Let $ 1\leq p <\infty$ and $p\ne 2$ and $\phi: M_{mn} \rightarrow M_{mn}$ be a linear map.
Then the following are equivalent.
\begin{enumerate}
\item[\rm (a)] $\|\phi(A\otimes B)\|_p = \|A\otimes B\|_p$ for all $A \in M_m$ and $B \in M_n$.

\item[\rm (b)] There are unitary matrices $U, V \in M_{mn}$ such that
$$\phi(A\otimes B) = U(\varphi_1(A) \otimes \varphi_2(B))V\quad\hbox{for all}\quad A \in M_m \hbox{ and } B \in M_n,$$
where $\varphi_s$ is the identity map or the transposition map $X \mapsto X^t$ for $s=1,2$.
\end{enumerate}
\end{theorem}

\begin{proof}
The implication (b) $\Rightarrow$ (a) is obvious. Conversely, assume that $\|\phi(A\otimes B)\|_p = \|A\otimes B\|_p$ for all $A \otimes B\in M_{mn}$. We first conclude  there exist unitary $U,V\in M_{mn}$ such that
\begin{equation}\label{eq.8}
\phi(E_{ii}\otimes E_{jj}) = U \left( E_{ii}\otimes E_{jj} \right)V \quad {\rm for}~  1\le i\le m  ~{\rm and}~  1\le j\le n.
 \end{equation}
Then we can adapt the arguments in the last three paragraphs in the proof of Theorem \ref{T1}
to verify that $\phi$ has the form claimed in (b).

As in the proof of Theorem \ref{T2.3}, to prove (\ref{eq.8}), it suffices to show that
\begin{equation}\label{eq.9}
\phi(E_{ii}\otimes E_{jj})\, \bot\, \phi(E_{rr}\otimes E_{ss})
\end{equation}
for any distinct pairs $(i,j)$ and $(r,s)$ with $1\leq i,r\leq m$ and $1\leq j,s\leq n$.

If $i=r$ or $j=s$, we have $$ \| \phi(E_{ii}\otimes E_{jj})+  \phi(E_{rr}\otimes E_{ss})\|_p^p+\| \phi(E_{ii}\otimes E_{jj})- \phi(E_{rr}\otimes E_{ss})\|_p^p= 2 \left( \|\phi(E_{ii}\otimes E_{jj}) \|_p^p+ \| \phi(E_{rr}\otimes E_{ss})\|_p^p \right).$$
Applying Lemma \ref{le3} and Remark \ref{rem7}, we get (\ref{eq.9}).
If $i\ne r$ and $j\ne s$, we have
\begin{multline*}
\| \phi(E_{ii}\otimes (E_{jj}+E_{ss}))+  \phi(E_{rr}\otimes (E_{jj}+E_{ss}))\|_p^p+\| \phi(E_{ii}\otimes (E_{jj}+E_{ss}))- \phi(E_{rr}\otimes(E_{jj}+E_{ss}))\|_p^p \cr
= 2 \left( \|\phi(E_{ii}\otimes (E_{jj}+E_{ss})) \|_p^p+ \| \phi(E_{rr}\otimes (E_{jj}+E_{ss}))\|_p^p \right).
\end{multline*}
By Lemma \ref{le3}, we have $\phi(E_{ii}\otimes (E_{jj}+E_{ss}))\,\bot\, \phi(E_{rr}\otimes (E_{jj}+E_{ss}))$.
With the fact that $\phi(E_{ii}\otimes E_{jj}) \,\bot\, \phi(E_{ii} \otimes E_{ss})$
and $\phi(E_{rr}\otimes E_{jj}) \,\bot\, \phi(E_{rr} \otimes E_{ss})$,
we conclude that (\ref{eq.9}) holds. This completes the proof.
\end{proof}

\begin{remark}\em
For $p=2$, i.e., the Frobenius norm case, the statements (a) and (b) in Theorem \ref{T2.5} are not equivalent. 
One can consider the linear map  $\psi(A)=\left[b_{ij} \right]$ for $A=\left[a_{ij}\right]\in M_{mn}$ such that $b_{1,mn}=a_{mn,1}$, $b_{mn,1}=a_{1,mn}$ and $b_{ij}=a_{ij}$ for $(i,j)\not\in \{(1,mn),(mn,1)\}$.
In fact, any linear map on $M_{mn}$ preserving the inner product $(A,B) = \tr(AB^*)$ on
$M_{mn}$ will preserve the Frobenius norm.
\end{remark}

\begin{remark} \em
Note that the maps in Theorems \ref{T1}, \ref{T2.3}, and \ref{T2.5} may not satisfy
$\|\phi(C)\| = \|C\|$ for all $C \in M_{mn}$. 
For instance, if $\phi(A\otimes B) = A \otimes B^t$ and if
$$C_r = r^2E_{11}\otimes E_{11} + r(E_{12}\otimes E_{12} + E_{21}\otimes E_{21})
+ E_{22}\otimes E_{22} \quad\hbox{for}\quad r \ge 0,$$
then $C_r$ has singular values $r^2+1, 0, \dots, 0$, and
$\phi(C_r)$ has singular values $r^2,r,r,1,0, \dots,0$.
Then for any positive number $r\ne 1$, 
$\|\phi(C_r)\| \ne \|C_r\|$
unless $\|\cdot\|$ is the Frobenious norm.
However, if we assume that $\phi$ satisfies
$\|\phi(C)\| = \|C\|$
for $C$ of the tensor form $A\otimes B$, and also for $C = C_r$ for some positive number $r\ne 1$,
then one easily deduces that
both $\varphi_i$ mentioned in the theorems have to be of the same type, i.e., both are identity map, or both
are the transposition map. It follows that there are unitary $U, V \in M_{mn}$
such that $\phi$ has the form
$$X \mapsto UXV \qquad \hbox{ or } \qquad X \mapsto UX^tV$$
and hence, $\|\phi(X)\| = \|X\|$ for all $X \in M_{mn}$.
\end{remark}

\section{Multipartite systems}

In this section we   extend the previous results  to multipartite systems $\M$, $m\ge 2$.
Let $A_i\in M_{n_i}, i=1,\ldots,m$.   We denote $ \bigotimes_{i=1}^mA_i=A_1\otimes A_2\otimes\cdots\otimes A_m$.

\begin{theorem}
\label{T4}
Let $1\leq k\leq \prod_{i=1}^m n_i$ and $\phi: \N\rightarrow \N$ be a linear map.
The following are equivalent.
\begin{enumerate}
\item[\rm (a)] $\|\phi(\A)\|_{(k)} = \|\A\|_{(k)}$ for all $A_i \in M_{n_i}$, $i = 1,\dots,m$.

\item[\rm (b)] There are unitary matrices $U, V \in M_{n_1\cdots n_m}$ such that
\begin{equation}\label{eq3.1}
\phi(\A) = U(\varphi_1(A_1)\otimes \cdots\otimes\varphi_m(A_m))V \quad \hbox{for all}\quad A_i \in M_{n_i},\ i = 1,\dots,m,
\end{equation}
where $\varphi_s$ is the identity map or the transposition map $X \mapsto X^t$ for $s=1,\ldots,m$.
\end{enumerate}
\end{theorem}

\begin{proof}
The sufficiency part is clear. To prove the necessity part, we use induction on $m$. By Theorem \ref{T1} and Theorem \ref{T2.3}, we already know that the statement
of Theorem \ref{T4}  is true for bipartite systems. So, assume that $m\ge 3$ and that the result holds for all $(m-1)$-partite systems.
We need to prove that the same is true for $m$-partite systems.

Denote $N = \prod_{i=1}^m n_i$. 
First we claim that there exist unitary $U,V\in M_{N}$ such that
\begin{equation}\label{eq3.11}
\phi(E_{j_1j_1} \otimes \cdots \otimes E_{j_mj_m}) = U ( E_{j_1j_1} \otimes \cdots\otimes  E_{j_mj_m} )V
\quad \hbox{for all}\quad 1\le j_i \le n_i.
\end{equation}

When $k=1$, as in the proof of Theorem \ref{T1}, we can successively consider 
\begin{multline*}
\phi(E_{i_1i_1}\otimes\cdots\otimes E_{i_{m-2},i_{m-2}}\otimes E_{i_{m-1},i_{m-1}}\otimes(E_{i_mi_m}+\gamma E_{j_mj_m})), \cr
\phi(E_{i_1i_1}\otimes\cdots\otimes E_{i_{m-2},i_{m-2}}\otimes(E_{i_{m-1}i_{m-1}}+\gamma E_{j_{m-1}j_{m-1}})\otimes D_m),\quad \ldots,\cr
\phi((E_{i_1i_1}+\gamma E_{j_{1}j_{1}})\otimes D_2\otimes \cdots\otimes D_{m-2} \otimes D_{m-1} \otimes  D_m)
\end{multline*}
to obtain (\ref{eq3.11}), where $D_2,\ldots, D_m$ are arbitrary diagonal matrices.

When $k\geq 2$, as in the proof of   Theorem \ref{T2.3}, it suffices to show
\begin{equation}\label{eq3.2}
\phi(E_{i_1i_1} \otimes \cdots\otimes E_{i_mi_m})\, \bot\,  \phi(E_{j_1j_1} \otimes \cdots\otimes E_{j_mj_m})
\quad \hbox{for any}\quad (i_1,\ldots,i_m)\ne (j_1,\ldots,j_m).
\end{equation}
To confirm (\ref{eq3.2}), it suffices  to verify that for any $1\le r \le m$,
\begin{multline}\label{eq3.3}
\hspace{-4mm}
\phi\left( \bigotimes_{u=1}^{r-1}(E_{i_u i_u}+E_{j_u j_u})
\otimes E_{i_r i_r} \otimes \bigotimes_{u=r+1}^m E_{i_u i_u}\right) 
\ \bot\ 
\phi\left( \bigotimes_{u=1}^{r-1}(E_{i_u i_u}+E_{j_u j_u})
\otimes E_{j_r j_r} \otimes \bigotimes_{u=r+1}^m E_{i_u i_u}\right)
\end{multline}
for any distinct $\bi = (i_1,\ldots,i_m)$ and $\bj = (j_1,\ldots,j_m)$ with $i_u\ne j_u$, $1\le u \le r$.
Denote by $G_r = G_r(\bi,\bj)$ and $\hat G_r = \hat G_r(\bi,\bj)$ the two matrices in (\ref{eq3.3}) accordingly.
We consider two cases.

\medskip
{\it Case 1.} For $r \le \log_2 k$, as
$\|  \alpha G_r+\beta \hat G_r \|_{(k)}
=|\alpha | \|  G_r   \|_{(k)}+|\beta| \| \hat  G_r  \|_{(k)}$
for all complex $\alpha$ and $\beta$.
Applying Lemma \ref{le1}, we get 
\begin{eqnarray*}
G_r\, \bot\, \hat G_r \quad\hbox{and}\quad \rank G_r + \rank \hat G_r \le k 
\quad\hbox{for all}\quad r \le \log_2 k.
\end{eqnarray*}
Now as $G_{s+1} = G_{s} + \hat G_{s}$ and $G_s \, \bot\, \hat  G_s$ for all $s \le \log_2 k$,
$$\|G_{s+1}\| \le \max\{ \|G_s\|, \|\hat G_s\|\} \le \max \left\{ \|G_s(\bi,\bj)\|:
i_u\ne j_u, 1\le u \le s \right\},$$
where $\|\cdot\|$ is the spectral norm. Hence, 
$$\max \left\{ \|G_{s+1}(\bi,\bj)\|:
i_u\ne j_u, 1\le u \le s+1 \right\}
\le \max \left\{ \|G_s(\bi,\bj)\|:
i_u\ne j_u, 1\le u \le s \right\}.$$
As the inequality holds for all $s \le \log_2 k$, it follows that
\begin{eqnarray*}
\|G_r\| \le 
\max \left\{ \|G_1(\bi,\bj)\|: i_1\ne j_1\right\} 
&=& \max \left\{ \| \phi(E_{i_1i_1} \otimes \cdots\otimes E_{i_mi_m}) \|:
(i_1,\dots,i_m) \right\} \cr
&\le& \max \left\{ \| \phi(E_{i_1i_1} \otimes \cdots\otimes E_{i_mi_m}) \|_{(k)}:
(i_1,\dots,i_m) \right\} = 1.
\end{eqnarray*}
Similarly, one conclude that $\|\hat G_r\| \le 1$.

\medskip
{\it Case 2.} For $r > \log_2 k$, we claim that $G_r$ and $\hat G_r$ are orthogonal and 
both of them have singular values $0$ and $1$ only. 
We prove the claim by induction on $r$.

Suppose $\log_2 k < r \le 1+\log_2 k$. Notice that 
$G_r = G_{r-1} + \hat G_{r-1}$. By Case 1, 
$$G_{r-1} \,\bot\, \hat G_{r-1},\quad
\|G_{r-1}\|\le 1,\quad 
\|\hat G_{r-1}\| \le 1, \quad\hbox{and}\quad
\rank G_{r-1} + \rank \hat G_{r-1}  \le k.$$ 
Therefore, $\|G_r\| \le 1$ and $\rank G_r \le k$. Since
$$\|G_r+\alpha \hat G_r\|_{(k)}=k
\quad\hbox{ and }\quad
\|  2G_r+\beta \hat G_r \|_{(k)}
=  \| G_r   \|_{(k)}+  \|G_r+\beta \hat G_r\|_{(k)},
$$
for any complex unit $\alpha$,
by Lemma \ref{le2}, we obtain $G_r \, \bot\, \hat G_r$. 
Moreover, $G_r$ has singular values $0$ and $1$ only.
Similarly, one can conclude that $\hat G_r$ has singular values $0$ and $1$ only.
Now assume that the claim holds for some $r > \log_2 k$.
We will show that the claim also holds for $r+1$.
By induction assumption and the fact that $G_{r+1} = G_r + \hat G_r$,
we conclude that $G_{r+1}$ has singular values $0$ and $1$ only.
The same conclusion holds for $\hat G_{r+1}$.
By Lemma \ref{2.5} and the fact that $\|G_{r+1} + \alpha \hat G_{r+1}\| = k$ for all complex unit $\alpha$,
we get $G_{r+1} \, \bot \, \hat G_{r+1}$.
Therefore the claim holds.

\medskip
Combining the above two cases, we see that (\ref{eq3.3}) holds and hence the statement (\ref{eq3.2}) follows.
Therefore, the claim (\ref{eq3.11}) holds for all $k$.
Without loss of generality, we may assume $U=V=I_N$ in (\ref{eq3.11}). Following a similar argument as in Theorems \ref{T1} and \ref{T2.3}, 
one can conclude that
$$\phi\left(E_{j_1j_1} \otimes \cdots \otimes E_{j_{m-1}j_{m-1}} \otimes B \right)
= E_{j_1j_1} \otimes \cdots \otimes E_{j_{m-1}j_{m-1}} \otimes \varphi_{j_1,\dots,j_{m-1}}(B)$$
for all $1\le j_i \le n_i$ with $1\le i \le m-1$ and $B\in M_{n_m}$,
where $\varphi_{j_1,\dots,j_{m-1}}$ can be assumed to be either the identity map or the transposition map.
Following the same argument, we can further conclude that
for any $X=X_1\otimes \cdots\otimes X_{m-1}$ with $X_i\in M_{n_i}$ being unitary for $1\leq i\leq m-1$,
there are unitary $U_X$ and $V_X$ such that 
$$\phi\left( \left( \bigotimes_{i=1}^{m-1} X_iE_{j_ij_i}X_i^* \right) \otimes B\right)
= U_X\left( \left( \bigotimes_{i=1}^{m-1} X_iE_{j_ij_i}X_i^* \right) \otimes  \varphi_{j_1,\dots,j_{m-1},X}(B) \right)V_X$$
for all $1\le j_i \le n_i$ with $1\le i \le m-1$ and $B\in M_{n_m}$,
where $\varphi_{j_1,\dots,j_{m-1},X}$ can be assumed to be either the identity map or the transposition map,
depending on $j_1,\dots,j_{m-1}$ and $X$. By the fact that $\phi(I_N) = I_N$, we have $V_X^* = U_X$.


Again, considering all symmetric $S\in   M_{n_m}$  as in the proof of Theorem \ref{T1}, we can show that there exists $W_X \in M_{n_1\cdots n_{m-1}}$ such that
$$\phi\left( \left( \bigotimes_{i=1}^{m-1} X_iE_{j_ij_i}X_i^* \right) \otimes B\right)
= W_X \left( \bigotimes_{i=1}^{m-1} X_iE_{j_ij_i}X_i^* \right) W_X^* \otimes  \varphi_{j_1,\dots,j_{m-1},X}(B) $$
for all $1\le j_i \le n_i$ with $1\le i \le m-1$ and $B\in M_{n_m}$.
Consider the linear map (known as the partial trace function in quantum information science context) on $M_N$
defined by $X \otimes Y \mapsto (\tr X)Y$ for $X \in M_{n_1\cdots n_{m-1}}$ and $Y \in M_{n_m}$
and apply to the above equation, we see that every choice of $\bigotimes_{j=1}^{m-1} X_i E_{j_ij_i}X_i^*$ 
gives rise to a linear map $\varphi_{j_1, \dots, j_m,X}$, which is either the identity map or the
transposition map. Evidently, the map 
$$\bigotimes_{j=1}^{m-1} X_i E_{j_ij_i}X_i^* \mapsto \varphi_{j_1, \dots, j_m,X}$$
is linear and hence continuous; the set
\begin{multline*}
\left\{ \bigotimes_{j=1}^{m-1} X_i E_{j_ij_i}X_i^*: 
1 \le j_i \le n_i \hbox{ and } X_i^*X_i = I_{n_i},\hbox{ for } i = 1, \dots, m-1 \right\} \cr
= \left\{x_1x_1^*\otimes \cdots \otimes x_{m-1}x_{m-1}^* : x_i \in \mathbb{C}^{ n_i} \hbox{ with } 1\le i \le m-1 \right\}
\end{multline*}
is connected. Thus, all the maps $\varphi_{j_1,\dots,j_{m-1},X}$ have to be the same.
Assume that this common map is $\varphi_m$, which is either the identity map or the transposition map.
By linearity, one can conclude that for any $\A \in \M $,  we have
$$\phi \left(\A \right) = \psi(A_1\otimes\cdots \otimes A_{m-1})\otimes \varphi_m(A_m)$$ for some $\psi(A_1\otimes\cdots \otimes A_{m-1}) \in M_{n_1\cdots n_{m-1}}$.
Note that $\psi:M_{n_1\cdots n_{m-1}}\rightarrow M_{n_1\cdots n_{m-1}}$ preserves the Ky Fan $k$-norm of all  matrices in $M_{n_1}\otimes\cdots\otimes M_{n_{m-1}}$. By the induction hypothesis, we know there exist unitary $\tilde{U },\tilde{V}$ such that
 $$\psi(A_1\otimes\cdots \otimes A_{m-1})=\tilde{U}(\varphi_1(A_1)\otimes\cdots\otimes \varphi_{m-1}(A_{m-1}))\tilde{V}$$
with each $\varphi_j$ being either the identity map or the transposition map. Hence, $\phi$ has the desired form and the proof is completed.
\end{proof}

\medskip
Using a similar argument and applying Lemma \ref{le3}, we can extend Theorem \ref{T2.5} to multipartite systems as follows.
\begin{theorem}
\label{T5}
Let $ 1\leq p <\infty$ and $p\ne 2$ and $\phi: \N\rightarrow \N$ be a linear map.
The following are equivalent.
\begin{enumerate}
\item[\rm (a)] $\|\phi(\A)\|_p = \|\A\|_p$ for all $A_i \in M_{n_i}$, $i =1,\dots,m$.

\item[\rm (b)] There are unitary matrices $U, V \in M_{n_1\cdots n_m}$ such that
\begin{equation*}\label{eq2}
\phi(\A) = U(\varphi_1(A_1)\otimes \cdots\otimes\varphi_m(A_m))V,\quad \hbox{for all}\quad A_i \in M_{n_i},\ i = 1,\dots,m,
\end{equation*}
where $\varphi_s$ is the identity map or the transposition map $X \mapsto X^t$ for $s=1,\ldots,m$.
\end{enumerate}
\end{theorem}

\bigskip
\noindent{\bf Acknowledgment}

This research was supported by a Hong Kong GRC grant PolyU 502411 with Sze as the PI.
The grant also supported the post-doctoral fellowship of Huang
and the visit of Fo\v{s}ner to the Hong Kong Polytechnic University in the summer of 2012.
She gratefully acknowledged the support and kind hospitality from the host university.
Li was supported by a GRC grant and a USA NSF grant; this research was done when he was
a visiting professor of the University of Hong Kong in the spring of 2012; furthermore,
he is an honorary professor of Taiyuan University of Technology (100 Talent Program scholar),
and an honorary professor of the  Shanghai University.



\baselineskip16pt

\pagebreak


\end{document}